\documentclass[11pt]{article}
\usepackage{amsmath,amsthm,amssymb,amsfonts}
\usepackage[width=16.5cm,height=23cm]{geometry}
\usepackage{enumerate}
\usepackage{mathrsfs}
\usepackage[cmtip,all]{xy}
\usepackage{tikz}

\usepackage{ifxetex}
\input ifxetex.sty
\usepackage{enumerate}
\usepackage{mathrsfs}
\usepackage[cmtip,all]{xy}
\usepackage{tikz}
\usetikzlibrary{shapes.geometric}
\usepackage[final]{hyperref}
\definecolor{sq3sq1color}{rgb}{0.5,0,0.5}
\definecolor{sq2color}{rgb}{0.1,0.7,0.1}
\colorlet{taucolor}{red}
\colorlet{partialsq2sq1color}{sq3sq1color!53!black}
\colorlet{incsq2sq1color}{sq3sq1color!67!green}
\colorlet{sq2rhosq1color}{taucolor!45!sq2color}
\colorlet{sq2prcolor}{sq2color!53!black}
\colorlet{incsq2color}{sq2color!67!yellow}
\colorlet{sq2partialcolor}{sq2color!42!blue}
\colorlet{tauprcolor}{taucolor!53!black}
\colorlet{taupartialcolor}{taucolor!42!yellow}

\usepackage{cleveref}
\RequirePackage[color       = blue!20,
                bordercolor = black,
                textsize    = tiny,
                figwidth    = 0.99\linewidth]{todonotes}

\theoremstyle{plain}
\newtheorem{theorem}{Theorem}[section]
\newtheorem*{theorem*}{Theorem}
\newtheorem{lemma}[theorem]{Lemma}
\newtheorem{proposition}[theorem]{Proposition}
\newtheorem*{proposition*}{Proposition}

\newtheorem*{corollary*}{Corollary}
\newtheorem{remark}[theorem]{Remark}
\newtheorem{remark*}{Remark}

\newtheorem{defn*}{Definition}

\newtheorem*{induction*}{Induction Hypothesis}

\newcommand{\im}{\operatorname{im}}
\newcommand{\supp}{\mathrm{supp}}

\newcommand{\CC}{\mathbb{C}}

\newcommand{\FF}{\mathbb{F}}

\newcommand{\MM}{\mathbb{M}}
\newcommand{\NN}{\mathbb{N}}

\newcommand{\QQ}{\mathbb{Q}}
\newcommand{\RR}{\mathbb{R}}

\newcommand{\ZZ}{\mathbb{Z}}

\newcommand{\Spec}{\operatorname{Spec}}

\newcommand{\smsh}{\wedge}

\newcommand{\mfz}[1][]{\MM^F_{#1}\ZZ}
\newcommand{\mfsz}{\MM^F_*\ZZ}

\newcommand{\mfsf}{\MM^F_*\FF_p}

\newcommand{\msz}[1][]{\MM^S_{#1}\ZZ}
\newcommand{\mssz}{\MM^S_*\ZZ}
\newcommand{\msf}[1][]{\MM^S_{#1}\FF_p}
\newcommand{\mssf}{\MM^S_*\FF_p}
\newcommand{\tmsz}[1][]{\widetilde{\MM^F_{#1}\ZZ}}
\newcommand{\tmssz}{\widetilde{\MM^S_*\ZZ}}
\newcommand{\tmsf}[1][]{\widetilde{\MM^F_{#1}\FF_p}}
\newcommand{\tmssf}{\widetilde{\MM^S_*\FF_p}}

\begin{document}

\title{The integral motivic dual Steenrod algebra}
\author{Bj{\o}rn Ian Dundas, Paul Arne {\O}stv{\ae}r}
\date{\today}
\maketitle

\begin{abstract}
  We calculate the motivic \emph{integral} dual Steenrod algebra $\mssz(\msz)$ over base schemes $S$ for which the \emph{mod $p$} motivic dual Steenrod algebra $\mssf(\msf)$ conforms with Voevodsky's formula.  Then
  \begin{quote}
    the kernel of the augmentation $\mssz(\msz)\to\mssz$ is simple $p$-torsion; 
  \end{quote}in more detail, after localizing at $p$
    the projection $\ZZ\to\FF_p$ gives rise to a pullback 
    square of commutative $\mssz$-algebras
    $$\xymatrix{
 \mssz(\msz)\ar[r]\ar[d]&\ker\beta\ar[d]\\\mssz\ar[r]&\mssf,}
  $$
   where  the $\mssf$-algebra $\ker\beta\subseteq \mssf(\msz)$ is the kernel of the Bockstein homomorphism.
   
 Concrete calculations follow, for instance, if $F$ is a field of characteristic different from $p$ where all $p$-torsion in the motivic homology $\mfsz$ of $\Spec(F)$ is $p$-divisible 
 (meaning that if $px=0$ then $x=py$ for some $y$ -- as is the case if $F$ is algebraically closed) then the motivic $p$-adic dual Steenrod algebra is expressed explicitly as an 
 algebra in terms of generators and relations.  Some other cases of interest are also explored more fully, most notably the base cases of finite fields and $\ZZ[1/2]$.
\end{abstract}

\section{Introduction}
\label{sec:intro}
As part of his proof of the Milnor and Bloch-Kato conjectures on Galoic cohomology and K-theory 
Voevodsky \cite{Voevodskymod2}, \cite{MR2811603} 
needed to know the motivic mod-$p$ dual Steenrod algebra $\mssf(\msf)=\pi_*(\msf\smsh\msf)$ and deduced the formula
\begin{equation}
  \mssf(\msf)=
\begin{cases}
  \mssf[\tau_0,\tau_1,\dots,\xi_1,\xi_2\dots]/\tau_i^2+\xi_{i+1}(\tau+\tau_0\rho)+\tau_{i+1}\rho&p=2\\
  \mssf[\tau_0,\tau_1,\dots,\xi_1,\xi_2\dots]/\tau_i^2&p>2,
\end{cases}
\end{equation}
where $\MM^S\FF_p$ is the mod $p$ motivic Eilenberg-Mac~Lane spectrum with associated graded commutative ring $\mssf=\pi_*\MM^S\FF_p$, $\rho\in\MM_{-1,-1}\FF_p$ is represented by $-1\in K_1\ZZ$ and $\tau\in\MM_{0,-1}\FF_p$.
At the outset, the base scheme $S$ was the spectrum of a field of characteristic zero, but the formula holds in greater generality; 
for instance, 
if $S$ is essentially smooth over a Dedekind domain where $p$ is invertible, see \cite{MR3730515} and \cite{spitzweck2013commutative}.

In this paper we consider motivic \emph{integral} homology $\MM^S\ZZ$
and explore the structure of the motivic integral dual Steenrod algebra $\MM^S_*\ZZ(\MM^S\ZZ)=\pi_*(\MM^S\ZZ\smsh\MM^S\ZZ)$.  Just as in the topological case explored by Kochman \cite{MR686130}, the integral case is significantly wilder than the mod $p$-case, but the kernel of the augmentation map $\mssz(\msz)\to\mssz$ is simple $p$-torsion and determined by the first Bocksteins from the mod $p$ dual Steenrod algebra.  Hence, the higher torsion in the K-theory of $S$, reflected in $\mssz$, affects the motivic integral dual Steenrod algebra only to a limited extent.  More precisely, if $S$ is so that Voevodsky's formula for the  motivic mod-$p$  dual Steenrod algebra holds, then
Theorem~\ref{cor:aspullback} gives the integral motivic Steenrod algebra as a pullback (after localizing at $p$) of commutative $\mssz$-algebras
    $$\xymatrix{
 \mssz(\msz)\ar[r]\ar[d]&\ker\beta\ar[d]\\\mssz\ar[r]&\mssf,}
$$
where  the $\mssf$-algebra $\ker\beta$ is the kernel of the Bockstein homomorphism on
$$\mssf(\msz)\cong
  \begin{cases}
    \mssf[\tau_1,\tau_2,\dots,\xi_1,\xi_2,\dots]/\tau_{i+1}^2+\xi_{i+2}\eta_L(\tau)+\tau_{i+2}\eta_L(\rho)& p=2\\
    \mssf[\tau_1,\tau_2,\dots,\xi_1,\xi_2,\dots]/\tau_{i+1}^2&p>2.
  \end{cases}
  $$

With this fact established, one can deduce concrete generators and relation expressions for the motivic integral dual Steenrod algebra, given that such a description is known for the integral motivic cohomology of $S$.  We discuss a few fundamental cases.

The simplest is the case where $S$ is an algebraically closed field of characteristic different from $p$ (which is a special case of Proposition~\ref{prop:algclosed} 
which applies when the $p$-torison in the integral motivic homology is $p$-divisible; that is, without simple $p$-torsion).
Removing the infinitely $p$-divisible part by considering $p$-adic coefficients -- which in this case is given by $\mssz_p=\ZZ_p[\tau]$ -- we get a description of the motivic $p$-adic dual Steenrod algebra $\mssz(\msz)_{p}$ as an augmented $\mssz_{p}$-algebra, 
  $$\mssz(\msz)_{p}\cong
  \mssz_{p}[y_{a,U}]/\mathcal I,$$
  where $a$ varies over functions $a\colon\NN\to\NN$ with finite support and $U$ over finite sets of positive integers and where the ideal $\mathcal I$ is
generated by
\begin{enumerate}
\item $py_{a,U}$,
\item For each $j>0$, $\sum_{U\subseteq\supp\, a, |U|=j}y_{a-\delta_U,U}$ (where $(a-\delta_U)_i$ is $a_i$ if $i\notin U$ and $a_{i-1}$ if $i\in U$) and 
\item $y_{a,U}y_{b,T}$ plus terms linear in the $y_{c,S}$s, as written out explicitly in Lemma~\ref{lem:relation}. 
\end{enumerate}
Comparing with the mod-$p$ case, the generator $y_{a,U}$ corresponds to
$$\sum_{j\in U}\xi_{j+1}\prod_{i\geq 0}\xi^{a_{i+1}}_{i+1}\prod_{k\in U-\{j\}}\tau_k\in\mssf(\mssf)$$ and the ideal $\mathcal I$ is essentially read off from this.
When working over complex numbers, the topological counterpart has no infinitely divisible classes. If one wants multiplicative information about the topological integral dual Steenrod algebra beyond what Kochman mentions explicitly, one replaces $\tau$ by $1$ in all the relations above.

To see what classes are generic for fields, we also calculate $\mssz(\msz)_{p}$ at the other extreme cases.  In positive characteristic, we must understand the case $S=\Spec(\FF_q)$ for $\FF_q$, a finite field of characteristic different from $p$.  This case is simplified by the fact that $K_2\FF_q=0$ and since $K_1\FF_q\cong\ZZ/(q-1)$ we get that $\mssz(\msz)_{p}$ depends on to what extent $p$ divides $q-1$.  If $p^2$ divides $q-1$, we are covered by  Proposition~\ref{prop:algclosed}, but if not, the calculation is performed in Lemma~\ref{lem:fifi}.

In characteristic zero, the real case is displayed in Lemma~\ref{lemma:R} and provides control over enough classes to do the base case of the $2$-inverted integers $\ZZ[1/2]$ in Lemma~\ref{lemma:Z12} when relying on the calculations of Spitzweck \cite{spitzweck2013commutative}.

While we were led to the motivic integral dual Steenrod algebra as part of our program of understanding motivic Hochschild homology initiated in \cite{dhko}, 
we expect the result to be useful for other calculations in the stable motivic homotopy category.
  
{\bf Acknowledgements.} 
The authors acknowledge support from the RCN Frontier Research Group Project no.~250399 ``Motivic Hopf Equations."
Work on this paper was supported by The European Commission -- Horizon-MSCA-PF-2022 ``Motivic integral $p$-adic cohomologies."

\subsubsection{Notation}
\label{sec:notation}

\begin{itemize}
\item $p$ is a prime, and $S$ is a base scheme with $p$ invertible, so the mod $p$ dual Steenrod algebra is given by formula~\ref{eq:VV} below,
\item we let $\eta_L\colon E\to E\smsh E'$ and $\eta_R\colon E'\to E\smsh E'$ be generic names for the left and right inclusions when $E$ and $E'$ are motivic ring spectra,
\item for $R$ a ring, $\MM^SR$ is the motivic commutative ring spectrum representing motivic cohomology with coefficients in $R$ and we write $\MM_*^SR=\pi_*\MM^SR$.  
The localization away from (resp. completion at) $p$ is denoted $\msz_{(p)}$ (resp. $\msz_p$) with coefficients $\mssz_{(p)}$ (resp. $\mssz_p$). 
\item For commutative motivic ring spectra $A$ and $B$, we write $A_*B=\pi_*(A\smsh B)$, and if $A=B$, we may refer to this graded commutative ring as the dual Steenrod algebra. 
By abuse of notation we write $\mssz(\msz)_p$ for $\pi_*((\msz\smsh\msz)_p)$, which we refer to as the $p$-adic motivic dual Steenrod algebra.
\end{itemize}

\section{The integral  motivic dual Steenrod algebra}
\label{sec:modpSteenrod}

For our main result, Theorem~\ref{cor:aspullback},  it will suffice to assume that the base scheme $S$ is such that Voevodsky's formula gives the mod $p$ dual Steenrod algebra
\begin{equation}
  \label{eq:VV}
  \pi_*(\msf\smsh\msf)=
\begin{cases}
  \mssf[\tau_0,\tau_1,\dots,\xi_1,\xi_2\dots]/\tau_i^2+\xi_{i+1}(\tau+\tau_0\rho)+\tau_{i+1}\rho&p=2\\
  \mssf[\tau_0,\tau_1,\dots,\xi_1,\xi_2\dots]/\tau_i^2&p>2
\end{cases}
\end{equation}
(see e.g.~\cite{riou2012operations} for elaborations on Voevodsky's arguments), where the $\mssf$-algebra structure is induced by the right inclusion $\eta_R\colon\msf\to\msf\smsh\msf$
so that when we write $\rho$ and $\tau$ in the formula above, we mean $\eta_R\rho$ and $\eta_R\tau$.  We will denote the unit element $\xi_0$ or $1$ according to what is most convenient.  Examples where Equation~\ref{eq:VV}  is true include the case where $S$ is essentially smooth over a field with characteristic different from $p$, see \cite{MR3730515}.

The conjugation $\chi\colon\msf\smsh\msf\cong\msf\smsh\msf$ gives a map of homotopy rings (but will change the $\mssf$-algebra structure from $\eta_R$ to the left unit map $\eta_L$) and is given inductively on generators by
\begin{align*}
  \chi(\rho)&=\eta_L\rho=\rho,     
  \\
  \chi(\tau)&=\eta_L\tau=\tau+\rho\tau_0,\\
  0&=\tau_0+\chi\tau_0\\
  0&=\tau_r+\chi\tau_r+\xi_1^{p^{r-1}}\chi\tau_{r-1}+\xi_2^{p^{r-2}}\chi\tau_{r-2}+\dots+\xi_{r-1}^p\chi\tau_1+\xi_r\chi\tau_0\\
  0&=\xi_r+\chi\xi_r+\xi_1^{p^{r-1}}\chi\xi_{r-1}+\xi_2^{p^{r-2}}\chi\xi_{r-2}+\dots+\xi_{r-1}^p\chi\xi_1\\
\end{align*}
for $r>0$.

The Bockstein on the dual Steenrod algebra is with respect to the map $p\smsh 1\colon \msz\smsh\msf\to \msz\smsh\msf$ and
\begin{align*}
  \beta\tau_r&=\xi_r,\quad r\geq 0,\\
  \beta\eta_R&\simeq 0.
\end{align*}

To see that $\beta\eta_R\simeq0$ consider the map of fiber sequence
$$\xymatrix{\msz\smsh\msf\ar[r]^{p\smsh 1}&\msz\smsh\msf\ar[r]^q&\msf\smsh\msf\ar[r]^{\partial}&\Sigma^{1,0}\msz\smsh\msf\\
\msf\ar[r]^{p\smsh 1}\ar[u]^{\eta_R}&\msf\ar[r]^-q\ar[u]^{\eta_R}&\msf\vee\Sigma^{1,0}\msf\ar[r]^-{\partial}\ar[u]&\Sigma^{1,0}\msf. \ar[u]^{\eta_R}}
$$

\begin{remark}
  In particular, note that $\beta\tau=0$ since we are using the right inclusion, as opposed to the left inclusion $\beta(\eta_L\tau)=\beta(\chi\tau)=\beta(\tau+\tau_0\rho)=\rho$.  While this may be surprising, note as a warning that taking the Bockstein of the relation $\tau_i^2+\xi_{i+1}(\tau+\tau_0\rho)+\tau_{i+1}\rho=0
$ better give $0$ on both sides, which would \emph{not} the case were $\beta\tau$ different from $0$.
Note that $\chi$ is a ring isomorphism, and so the $\chi\tau_i$ and $\chi\xi_{i+1}$ form a set of generators for $\mssf$, \emph{both} as a left and as a right $\mssf$ algebra (but one should beware that the relation one should quotient out by in the case $p=2$ now reads
$\chi\tau^2_i=\chi\xi_{i+1}\tau+\chi\tau_{i+1}\rho$).
\end{remark}

\begin{lemma}
  The map $q\colon\msz\to \msf$ induces an injection $q\colon\mssz(\msf)\to \mssf(\msf)$ with image the right $\mssf$-subalgebra generated by $\chi\tau_{i+1}$ and $\chi\xi_{i+1}$ for $i\geq 0$: consequently we get isomorphisms
  $$\mssz(\msf)\cong
  \begin{cases}
    \mssf[\chi\tau_{i+1},\chi\xi_{i+1}]_{i\geq 0}/(\chi\tau_{i+1})^2+\chi\xi_{i+2}\tau+\chi\tau_{i+2}\rho& p=2\\
    \mssf[\chi\tau_{i+1},\chi\xi_{i+1}]_{i\geq 0}/(\chi\tau_{i+1})^2&p>2
  \end{cases}
  $$
  as right $\mssf$-algebras.  Upon conjugation, we thus get an isomorphism
  $$\mssf(\msz)\cong
  \begin{cases}
    \mssf[\tau_{i+1},\xi_{i+1}]_{i\geq 0}/\tau_{i+1}^2+\xi_{i+2}\eta_L(\tau)+\tau_{i+2}\eta_L(\rho)& p=2\\
    \mssf[\tau_{i+1},\xi_{i+1}]_{i\geq 0}/\tau_{i+1}^2&p>2.
  \end{cases}
  $$
  of left $\mssf$-algebras.
\end{lemma}
\begin{proof}
  The map $q$ is an injection since $p=0\in \mssz(\msf)$.
  Since $\beta\chi\tau_0=-\beta\tau_0=-1$ induction on the equations for $\chi$ proves that $\beta\chi\tau_{i+1}=0$ and $\beta\chi\xi_{i+1}=0$ for $i\geq 0$.
  The statement follows since $q$ being an injection forces the kernels of $\partial$ and $\beta\Sigma^{1,0}q\partial$ to coincide.
\end{proof}
Now, to get at $\mssz(\msz)$ we consider
the map of cofiber sequences
$$
\xymatrix{\msz\smsh \msz\ar[r]^{p\smsh 1}\ar[d]^{1\smsh q}&\msz\smsh \msz\ar[r]^q\ar[d]^{1\smsh q}&\msf\smsh \msz\ar[r]^-{\bar\beta}\ar[d]^{1\smsh q}&\Sigma^{1,0} \msz\smsh \msz\ar[d]^{1\smsh q}\\
  \msz\smsh \msf\ar[r]^{p\smsh 1}&\msz\smsh \msf\ar[r]^q&\msf\smsh \msf\ar[r]^-{\bar\beta}&\Sigma^{1,0} \msz\smsh \msf.}
$$
Since (the vertical) $1\smsh q\colon\msf\smsh\msz\to \msf\smsh\msf$ identifies $\mssf(\msz)$ as a subalgebra of $\mssf(\msf)$, the nontrivial Bocksteins of the upper row are given by $\beta\tau_{i+1}=\xi_{i+1}$ for $i\geq0$, plus whatever Bocksteins that may be hiding in the coefficients $\mssf$ themselves.

When varying $a$ over functions $\NN\to\NN$ with finite support (with value $a_i$ at $i\in \NN$) and $U$ over finite sets of positive integers, the elements
\begin{equation}
  \label{eq:eta}
  \eta_{a,U}=\prod_{i\geq0}\xi^{a_{i+1}}_{i+1}\prod_{j\in U}\tau_j
\end{equation}
(where $\eta_{0,\emptyset}=1$) give a basis for $\mssf(\msz)$ as a free $\mssf$-module.

Let
$$X=\FF_p[\tau_{i+1},\xi_{i+1}]_{i\geq 0}/\tau^2_{i+1}=\FF_p\{\eta_{a,U}\}_{a,U}$$
equipped with the differential $\beta\eta_{a,U}=\sum_{j\in U}\eta_{a+\delta_j,U-\{j\}}$.  The
(purely algebraic) isomorphism
$$\mssf(\msz)\cong \mssf\otimes_{\FF_p} X$$ of $\mfsf$-modules (algebras if $p$ is odd) is also an isomorphism of Bockstein complexes.  
Hence the K\"unneth isomorphism yields an isomorphism of homology groups 
$$H_*(\mssf(\msz),\beta)\cong H_*(\mfsf,\beta)\otimes_{\FF_p} H_*(X,\beta).
$$
Since $H_*(X,\beta)\to\FF_p$ turns out to be an isomorphism, this gives full control of the Bockstein homology (also at $p=2$):

\begin{lemma}
  \label{lem:bockseteinxtildeiszero}
  The multiplication $\msf\smsh\msz\to\msf$ induces an isomorphism in Bockstein homology $H_*(\mssf(\msz),\beta)\cong\mfsf$.
\end{lemma}
\begin{proof}
  We must show that the augmentation $X\to\FF_p$ is an isomorphism in Bockstein homology.
  For $i\geq0$ and $j>0$ consider the acyclic complex
$$C^{i,j}=\{C^{i,j}_0\gets C^{i,j}_1\}=\{\FF_p\{\xi^{j}_{i+1}\}\gets\FF_P\{\xi^{j-1}_{i+1}\tau_{i+1}\}\}$$
given by $\beta$ and, in addition, for all $i\geq0$, let $C^{i,0}=\{\FF_p\{1\}\}$ be concentrated in degree $0$.

Note that for $\epsilon\in\ZZ/2$, $j>0$ and $i\geq0$ we have that 
$$C^{i,j+\epsilon}_\epsilon=\xi^j_{i+1}\tau^\epsilon_{i+1},
$$
and so the $\FF_p$-linear basis for $X$ takes the form
$$\eta_{a,U}=\prod_{i\geq0}C^{i,a_i+\epsilon_i}_{\epsilon_i},
$$
where $a\colon\NN\to\NN$ has finite support, $U$ is a finite set of positive integers and where $\epsilon\colon\NN\to\ZZ/2$ is the characteristic function with nonzero values $\epsilon(i)=1$ when $i+1\in U$.

In other words, $X$ is isomorphic to the tensor product $\bigotimes_{i\geq0,j\geq0} C^{i,j}$ (which is a filtered colimit of finite tensors), where all factors but the ones with $j=0$ are acyclic. The factors with $j=0$ are copies of $\FF_p$ concentrated in degree zero.
\end{proof}

Let $\tmsz\smsh\msz$ (resp. $\tmsf\smsh\msz$) be the fiber of the multiplication $\msz\smsh\msz\to\msz$
(resp. $\msf\smsh\msz\to\msf$) with homotopy ideal $\tmssz(\msz)\subseteq \mssz(\msz)$ (resp. $(\tmssf)\msz\subseteq (\mssf)\msz$).
By Lemma~\ref{lem:bockseteinxtildeiszero}, the Bockstein complex of $\tmssf(\msz)\subseteq \mssf(\msz)$ is acyclic, and so
\begin{theorem}
  \label{cor:aspullback}
  If the base scheme $S$ is so that the mod $p$ motivic dual Steenrod algebra is as in Equation~\ref{eq:VV} (for instance, if $S=Spec(F)$ for $F$ a field of characteristic different from $p$), then the square
  $$\xymatrix{
 \mssz(\msz)\ar[r]^-q\ar[d]&\ker\beta_{\mssf(\msz)}\ar[d]\\\mssz\ar[r]^q&\mssf}
$$
of commutative $\mssz$-algebras 
is a pullback after localizing at $p$, where $\ker\beta_{(\mssf)\msz}\subseteq (\mssf)\msz$  is the kernel of the Bockstein homomorphism.

   Otherwise stated, the augmentation ideal $\tmssz(\msz)$ is simple $p$-torsion, mapping isomorphically to the image of the Bockstein on $\tmssf(\msz)$ and so
  $\mssz(\msz)_{(p)}$ is the semi-direct product of $\mssz_{(p)}$ and the image of the Bockstein on $\tmssf(\msz)$, with module structure factoring over $q\colon\mssz_{(p)}\to\mssf$.
\end{theorem}
\begin{proof}
  Lemma~\ref{lem:bockseteinxtildeiszero} claims that $\im\beta_{\tmssf(\msz)}=\ker\beta_{\tmssf(\msz)}$, which is equivalent to $\tmssz(\msz)$ being simple $p$-torsion mapped isomorphically to $\ker\beta_{(\tmssf)\msz}$,
    so we only need to consider the last statement.  The map $\mssz\to\mssf$ factors over $\ker\beta_{\mssf}$, so pulling back the split exact sequence
  $$\ker\beta_{\tmssf(\msz)}\to \ker\beta_{\mssf(\msz)}\to\ker\beta_{\mssf}$$ gives that
  $\mssz(\msz)\cong \mssz_{p}\times_{\ker\beta_{\mssf}}\ker\beta_{\mssf(\msz)}$
  and observe that pulling back over $\ker\beta_{\mssf}$ or over $\mssf$ gives the same result.
\end{proof}

 The kernel $\tilde X$ of the augmentation $X\to\FF_p$ is a subcomplex of $X$ (since $1$ is not hit by a Bockstein) under the Bockstein differential, with vanishing homology
 $H_*(\tilde X,\beta)=0$.

 In order to get at the kernel of the Bockstein on $(\tmssf)\msz\cong \mssf\otimes\tilde X$, we should study $\tilde X$ in more detail: that is, we need to understand the algebraically closed case.

\section{The algebraically closed case}
\label{sec:algclosed}

In this section, 
we give a generator-and-relations expression for the integral motivic dual Steenrod algebra in the case where all $p$-torsion in $\mssz$ is $p$-divisible (and Equation~\ref{eq:VV} holds).
We start out with the situation where $S=\Spec(F)$ where $F=\bar F$ is an algebraically closed field and use the abbreviations $H=\MM^{\bar F}\FF_p$ and $M=\MM^{\bar F}\ZZ_p$, so that $H_*=\FF_p[\tau]$ and $M_*=\ZZ_p[\tau]$ and the Bockstein is trivial \cite{zbMATH01526539}.  In this case, the calculation of the $p$-adic dual Steenrod algebra parallels the topological calculation of Kochman \cite{MR686130}

If we let $B^\beta=\im\beta_{H_*M}$ and $Z^\beta=\ker\beta_{H_*M}$  be the Bockstein boundaries and cycles, so that $Z^\beta/B^\beta=H_*=H_*(H_*M,\beta)$ we see that the elements
$$y_{a,U}=\beta\eta_{a,U}=\sum_{j\in U}\eta_{a+\delta_j,U-\{j\}}$$
(where the $\eta$s were defined in Equation~\ref{eq:eta})  generate $B^\beta$ as an $H_*$-module.
To describe a basis, we use the isomorphism
$$H_*M\cong\bigoplus_a\bigotimes_{i\geq 0}C^{i,a_i}$$
where $a\colon\NN\to\NN$ has finite support, and so it is enough to find bases in each degree for each of the summands
$$B^{\beta,a}=B^\beta\cup \otimes_iC^{i,a_i}.$$
Degree $j$ is generated by the $\eta_{a-\delta_U,U}$, where $U\subseteq\supp\, a$ has $j$ elements (where $\delta_U=\sum_{j\in U}\delta_j$) and $B^{\beta,a}_{j-1}$ is generated by the corresponding $y_{a-\delta_U,U}$ with
$$0=\sum_{U\subseteq\supp\, a, |U|=j}y_{a-\delta_U,U}.
$$
In consequence
$$\{y_{a,U}\}_{\max\supp\, a\leq \max U}$$
gives a basis for the free $H_*$-module
$B^\beta$.  
Explicitly stated,
\begin{lemma}
  The $H_*$-module $B^\beta$ is free on the set consisting of the elements of the form
  $$\beta(\xi_1^{a_1}\dots\xi^{a_{n}}_{n}\tau_1^{\epsilon_1}\dots\tau_{n-1}^{\epsilon_{n-1}}\tau_{n}),$$ where $a_i\in\NN$ and $\epsilon_i\in\ZZ/2$.  
\end{lemma}
However, to state some formulas more easily, it is convenient to refer to the larger generating set and remember only the linear relation at the end.

\begin{lemma}
  \label{lem:relation}
  The graded commutative $H_*$-algebra $Z^\beta$ is subject to the following relations.
  When $p=2$ 
  $$y_{a,U}y_{b,T}=\sum_{k\in U}\tau^{|(U-\{k\})\cap T|}\,y_{a+b+\delta_{k-1}+\delta_{(U-\{k\})\cap T},(U-\{k\})\odot T}
    $$
    where $U\odot T=U\cup T-U\cap T$ is exclusive or.
    
    When $p$ is odd 
    $$y_{a,U}y_{b,T}=
    \begin{cases}
      \sum_{k\in T}\sigma(U+\{k\},T-\{k\})\,y_{a+b+\delta_{k-1},S\cup T-\{k\}}&\text{if }U\cap T=\emptyset\\
      \sigma(U-\{k\},T-\{k\})\,y_{a+b+\delta_{k-1},S\cup T-\{k\}}&\text{if }U\cap T=\{k\}\\
      0&\text{if } |U\cap T|>1,
    \end{cases}
    $$
    where $\sigma(U,T)$ is the sign of the shuffle needed to pass from the disjoint $U$ and $T$ into the naturally ordered union $U\cup T$.
\end{lemma}
\begin{proof}
  First note that for $p=2$ we have that $\eta_{a,U}\eta_{b,T}=\tau^{|U\cap T|}\eta_{a+b+\delta_{U\cap T},U\odot T}
  $ so that
  \begin{align*}
    y_{a,U}y_{b,T}&=\beta(\eta_{a,U})\beta(\eta_{b,T})=\beta(\beta(\eta_{a,U})\eta_{b,T})\\
    &=
      \beta(\sum_{k\in U}\eta_{a+\delta_{k-1},U-\{k\}}\eta_{b,T})\\
    &=
      \sum_{k\in U}\tau^{|(U-\{k\})\cap T|}\beta(\eta_{a+b-\delta_{k-1}+\delta_{(U-\{k\})\cap T},(U-\{k\})\odot T}).
  \end{align*}
  For odd $p$,
  $$\eta_{a,U}\eta_{b,T}=
  \begin{cases}
    \sigma(U,T)\eta_{a+b,U\cup T}&\text{if } U\cap T=\emptyset\\
    0&\text{otherwise.}
  \end{cases}
  $$
  Using $y_{a,U}y_{b,T}=\beta(\eta_{a,U}\beta(\eta_{b,T}))=\beta(\eta_{a,U}\sum_{k\in T}\eta_{b+\delta_{k-1},T-\{k\}})$ and tracking signs we get the desired result.
\end{proof}

\begin{proposition}
  \label{prop:algclosed}
\label{sec:trivbock}
If  all $p$-torsion in $\mssz_{p}$ is $p$-divisible and Equation~\ref{eq:VV} holds, for instance, if $S=\Spec(F)$ with $F$ an algebraically closed field (so that there is no torsion), then
the $p$-adic motivic dual Steenrod algebra $\mssz(\msz)_{p}$ is isomorphic as an augmented $\mssz_{p}$-algebra to 
  $$
  \mssz_{p}[y_{a,U}]/\mathcal I,$$
  where $a$ varies over functions $a\colon\NN\to\NN$ with finite support and $U$ over finite sets of positive integers and where the ideal $\mathcal I$ is
generated by
\begin{enumerate}
\item $py_{a,U}$,
\item $\sum_{U\subseteq\supp\, a, |U|=j}y_{a-\delta_U,U}$ and 
\item the relations of Lemma~\ref{lem:relation}. 
\end{enumerate}
    The degree of $y_{a,U}$ is
  $$|y_{a,U}|=(-1,0)+\sum_{i>0}a_i(p^i-1)(2,1)+\sum_{j\in U}(2p^j-1,p^j-1).$$
  The kernel of the augmentation $$\mssz(\msz)_{p}\to \mssz_{p}$$ is a free $\mssf$-module on the generators
  $\{y_{a,U}\}_{\max\supp\, a\leq \max U}$.
\end{proposition}
\begin{proof}
  For algebraically closed $F$, this follows from the discussion above.
  More generally, if all $p$-torsion  in $\mssz_{p}$ is $p$-divisible, then the $p$-Bockstein on $\mfsf$ is trivial, implying that the image of the Bockstein on 
  $(\tmssf)\msz\cong \mssf\otimes\tilde X$ is equal to $\mssf$ tensor the image of the Bockstein on $\tilde X$, reducing our calculation to the algebraically closed case.
\end{proof}

\section{Examples: other base schemes}
\label{sec:other}

In this section, we calculate the $p$-adic dual Steenrod algebra $\mssz(\msz)_{p}$ for some base schemes where $\mssz_{p}$ is known.

Assume given $\FF_p$-bases $B_H\subseteq Z_H$ for the image and kernel of the Bockstein $\beta$ on $\mssf$. Let $R_H\subseteq\beta^{-1}B_H\subseteq\mssf$ be a subset mapped bijectively to $B_H$ by $\beta$.  For $X=\FF_p\otimes_{\FF_p[\tau]}(\mssf)\msz$, we have the corresponding subsets $R_X=\{\eta_{a,U}\}$, $B_X=\{y_{a,U}\}$ and $Z_X=B_X\cup\{1\}$, so that $R_X\cup Z_X$ is an $\FF_p$-basis for $X$.

Letting
$$Z=\{z_H\otimes z_X\}_{z_H\in Z_H,z_X\in Z_X}\quad\text{ and }\quad U=\{\beta r\otimes \eta_{a,U}-(-1)^{d_r}r\otimes y_{a,U}\}_{r\in R_H},$$
(with $d_r$ being the total degree of $r$) we get that $Z\cup U\subseteq(\mssf)\msz$ is an $\FF_p$-basis for $\ker\beta_{(\mssf)\msz}$, and by Theorem~\ref{cor:aspullback} the only thing missing for having a complete description of $\mssz(\msz)_{p}$ is knowledge about the kernel of the $\mssf$-Bockstein.  
\subsection{The real numbers}
\label{sec:reals}
When $F$ is the field $\RR$ of real numbers and $p=2$ we have that  $\MM^\RR_*\FF_2=\FF_2[\rho,\tau]$ with $|\rho|=(-1,-1)$, $|\tau|=(0,-1)$ and $\beta\tau=\rho$.  Hence the kernel of the Bockstein is $\ker\beta_H=\FF_2[\rho,\tau^2]$, the image is the ideal $\rho\FF_2[\rho,\tau^2]$ to which $\tau\FF_2[\rho,\tau^2]$ mapped isomorphically and
$\MM^\RR\ZZ_2=\ZZ_2[\rho,\tau^2]/2\rho$.
Consequently,
\begin{lemma}
  \label{lemma:R}
  Over the reals, the $2$-adic dual Steenrod algebra $\MM^\RR_*\ZZ(\MM^\RR\ZZ)_2$ is generated by the elements $y_{a,U}$ and $\rho\eta_{a,U}+\tau y_{a,U}$ for varying $a$ and $U$ as an $\MM^\RR\ZZ_2=\ZZ_2[\rho,\tau^2]/2\rho$-subalgebra of
$$\MM^\RR_*\ZZ_2\times_{\MM^\RR_*\FF_2}(\MM^\RR_*\FF_2)\MM^\RR\ZZ=\ZZ_2[\rho,\tau^2]\times_{\FF_2[\rho,\tau]}\FF_2[\rho,\tau,\tau_{i+1},\xi_{i+1}]_{i\geq 0}/(\tau^2_{i+1}+\xi_{i+2}\tau+\tau_{i+2}\rho).$$
\end{lemma}
If one wants to write $\MM^\RR_*\ZZ(\MM^\RR\ZZ)_2$ as an explicit algebra, one uses the relations of Lemma~\ref{lem:relation}.

\begin{remark}
  When $p$ is odd, see \cite{stahn2021motivic}, 
  then $\MM^{\RR}_*\FF_p=\FF_p[\theta]$ with the element $\theta$ of bidegree $(0,-2)$ mapping to $\tau^2$ under the inclusion $\RR\subseteq\CC$. 
  For degree reasons, there is no room for Bocksteins, so $\MM^{\RR}_*\ZZ_p=\ZZ_p[\theta]$ and Proposition~\ref{sec:trivbock} yields that $(\MM^\RR_*\ZZ_p)\MM^\RR\ZZ\cong
  \ZZ_p[\theta][y_{a,U}]/\mathcal I$ (with the ideal $\mathcal I$ as described).
 \end{remark}
\subsection{Vanishing higher Milnor K-groups}
\label{sec:fifi}
The Milnor K-theory of a finite field $\FF_q$ vanishes in degrees greater than $1$ (since $K_2^M(\FF_q)=0$):
$K^M_*\FF_q=\ZZ[\epsilon]/((q-1)\epsilon, \epsilon^2)$ 
   (with $\epsilon$, a generator of the cyclic group $\FF_q^*$, is in degree $1$). In contrast, the Quillen K-theory is
   $K_*\FF_q=\ZZ[e_i]_{0< i}/((q^{i}-1)e_i,e_ie_j)$,  where $e_i$ is of degree $2i-1$ and with $e_1$ corresponding to $\epsilon\in K^M_1\FF_q$.
Furthermore, for $p$ not dividing $q$,  
we get that 
$$q\colon\MM^{\FF_q}_*\ZZ=\ZZ[\epsilon_i]_{0< i}/((q^{i}-1)\epsilon_i,\epsilon_i\epsilon_j)\rightarrow \MM^{\FF_q}_*\FF_p=\FF_p[\tau,\epsilon]/((q-1)\epsilon,\epsilon^2)
$$
sends $\epsilon_{i+1}$ to $\epsilon\tau^i$, with $\epsilon_i$ of bidegree $(-1,-i-1)$.

The Bockstein $\beta\colon\MM^{\FF_q}_*\FF_p\to \MM^{\FF_q}_*\FF_p$ is zero in all cases except when $p$, but not $p^2$, divides $q-1$.
Hence  Proposition~\ref{prop:algclosed} covers all situations but this exception, which is easily dealt with by the following observation, noting that $\beta(\tau^i)=i\epsilon\tau^{i-1}$:

\begin{lemma}
  \label{lem:fifi}
  Let $F$ be a field of characteristic different from $p$, with a primitive $p$th root of unity and no primitive $p^2$-root of unity and with $K^M_jF/p=0$ for $j>1$.  Then
  $$\mfsf=(\FF_p\ltimes \Sigma^{-1,-1}F^*/pF^*)[\tau],$$
  $$\ker\beta_{\mfsf}=\FF_p[\tau^p]\ltimes F^*/pF^*[\tau]
  $$
  (square zero extensions), and the set
  $\{\beta(\tau^{i})\}_{i\neq 0\mod p}$ is an $\FF_p$-basis for the image of the Bockstein, so that
 $\mfsz(\mfz)_{(p)}$ is the $\mfsz_{(p)}$-subalgebra of $\mfsz_{(p)}\times_{\mfsf}\times(\mfsf)\mfz$ generated by the set
  $$\{\tau^iy_{a,U}+i\beta(\tau)\cdot\tau^{i-1}\eta_{a,U}\}_{0\leq i<p, a, U}.$$
\end{lemma}
Again, explicit multiplicative relations are given by Lemma~\ref{lem:relation}.

\subsection{The $2$-integers}
\label{sec:rational}

Since the higher Milnor K-groups of finite fields vanish, Proposition~\ref{prop:algclosed} and Lemma~\ref{lem:fifi} applies and give rise to classes in the dual Steenrod algebra $\MM_*^S\ZZ(\MM^S\ZZ_p)$ for schemes $S$ of positive characteristic different from $p$. For the characteristic zero case, we should consider the base case $\Spec(\ZZ[1/p])$.

We focus on the case $p=2$ and let $S=\Spec(\ZZ[1/2])$.
Its mod $2$ motivic cohomology is given by
$$\MM^S_*\FF_2=\FF_2[\tau,\rho,\epsilon]/(\epsilon\rho,\epsilon^2),
$$
where the classes $\epsilon,\rho\in\MM_{-1,-1}^S\FF_2=K_1(\ZZ[1/2])/2$ are represented by the units $2$ and $-1$ respectively and the relations $\epsilon\rho=0$ and $\epsilon^2=0$ come from bilinearity and the Steinberg relations. ($\rho\epsilon=\{-1,2\}=\{1-2,2\}=0$ and $\epsilon^2=\{2,2\}=\{2,-2\}+\{2,-1\}=0+0$). That no powers of $\rho$ vanish follows by comparison with 
$\MM^\RR_*\FF_2$ since there is no $\tau$-torsion, see e.g., \cite[Corollary~7.10]{MR4192569}.

Using that $\MM_*^S\FF_2\to\MM^\QQ_*\FF_2$ is an injection of rings, Voevodsky's formula \eqref{eq:VV} for the dual Steenrod algebra $\MM_*^S\FF_2(\MM^S\FF_2)$ applies to $S=\Spec(\ZZ[1/2])$  since 
Spitzweck \cite[Theorem~11.24]{spitzweck2013commutative} has verified it on the level of $\MM^S_*\FF_2$-modules.

The Bockstein is generated by $\beta\tau=\rho$, and so, the kernel of the Bockstein on $\MM^S_*\FF_2$ is the subalgebra generated by $\epsilon$, $\rho$, $\tau^2$ and $\epsilon\tau$ and has $\FF_2$-basis $Z_H=\{\rho^i\tau^{2j},\epsilon\tau^k\}_{i,j,k\geq 0}$.
Choosing the basis $B_H=\{\rho^i\tau^{2j}\}_{i>0,j\geq 0}$ for the image of the Bockstein and letting $R_H=\{\rho^i\tau^{2j+1}\}_{i,j\geq 0}$ we obtain, as above, an $\FF_2$-basis for the kernel of the Bockstein on
$$\MM_*^S\FF_2(\MM^S\ZZ)=\FF_2[\tau,\rho,\epsilon,\tau_{i+1},\xi_{i+1}]_{i\geq 0}/(\epsilon\rho,\epsilon^2,\tau_{i+1}^2+\tau\xi_{i+2}+\rho\tau_{i+2}),$$
namely
$$Z\cup U=\{\rho^i\tau^{2j}, \epsilon\tau^i, \rho^i\tau^{2j} y_{a,U}, \epsilon\tau^i y_{a,U}, \rho^i\tau^{2j}(\rho\eta_{a,U}+\tau y_{a,U}) \}_{i,j\geq 0}.$$
More compactly, 
$\ker\beta_{\MM_*^S\FF_2(\MM^S\ZZ)}$ is the $\ker\beta_{\MM^S_*\FF_2}=\FF_2[\tau^2,\rho,\epsilon,\epsilon\tau]/(\epsilon\rho,\epsilon^2)$-subalgebra of $\MM_*^S\FF_2(\MM^S\ZZ)$ generated by
$\{y_{a,U}, \rho\eta_{a,U}+\tau y_{a,U}\}$ (here the relations $\epsilon\tau\cdot\rho=0$, $\epsilon\tau\cdot\epsilon=0$ and $\epsilon\tau\cdot\epsilon\tau=0$ are meant to be clear from $\epsilon\rho=0$ and $\epsilon^2=0$).

In order to fit this into a complete description of the $2$-adic dual Steenrod algebra of $\ZZ[1/2]$, we need an explicit calculation of the $2$-adic motivic cohomology algebra itself:

\begin{lemma}
  The $2$-adic motivic cohomology algebra of $\ZZ[1/2]$ is given by
  $$\MM^{\ZZ[1/2]}_*\ZZ_2=
\ZZ_2[\rho_{2i+1},\epsilon_{j+1}]_{i,j\geq0}/\left(
  \begin{smallmatrix}
    2\rho_{2i+1},\quad w_{2i}\epsilon_{2i},\\
    \rho_{2i+1}\epsilon_j,\quad\epsilon_i\epsilon_j,\\
    \rho_{2i+1}\rho_{2j+1}+\rho_1\rho_{2(i+j)+1}
  \end{smallmatrix}
\right),
$$
with
$w_i$ the exponent of the $e$-invariant \cite[p.~516]{MR3076731},
$|\rho_i|=|\epsilon_i|=(-1,i)$, and where the natural map $q\colon\MM^S_*\ZZ_2\to\MM^S_*\FF_2$ sends $\rho_{2i+1}$ to $\rho\tau^{2i}$ and $\epsilon_{i+1}$ to $\epsilon\tau^{i}$.
Hence, the image of $q$ is the subalgebra of $\MM_*^S\FF_2$ generated by $\rho\tau^{2i}$ and $\epsilon\tau^i$ to which $q$ defines an isomorphism from the quotient ring
$$\MM^{\ZZ[1/2]}_*\ZZ_2/\ker q=\FF_2[\rho_{2i+1},\epsilon_j]_{i,j\geq0}/(\rho_{2i+1}\epsilon_j, \epsilon_i\epsilon_j, \rho_{2i+1}\rho_{2j+1}+\rho_1\rho_{2(i+j)+1}).$$
\end{lemma}
\begin{proof}
  Set $S=\Spec(\ZZ[1/2])$.
  As abelian groups, by 
  \cite[Remark~46]{cwKhandbook}, we have in degree zero
$$\MM_{0,i}^S\ZZ_2 =
\begin{cases}
  \ZZ_2\{1\}&i=0\\
  0&i<0
\end{cases}
$$
while in degree $-1$ we have
$$\MM^S_{-1,k}\ZZ_2=
\begin{cases}
  \ZZ/2\{\rho_{k}\}\oplus\ZZ_2\{\epsilon_{k}\}&k\text{ odd}\\
  \ZZ/w_{k}\{\epsilon_{k}\}&k\text{ even}
\end{cases}
$$
(where $w_0=1$ by convention so that $\MM_{-1,0}\ZZ_2=0$).

  Since we know the Bockstein for $\MM^S_*\FF_2$ we deduce what the map $\MM^S_{-1,*}\ZZ_2\to\MM^S_{-1,*}\FF_2$ does to the generators: in even weight $-2i<0$, consider
  $$\xymatrixcolsep{.5pc}\xymatrix{
    \MM^S_{-2,-2i}\FF_2\ar@{=}[d]&\ar[l]_{q}\MM^S_{-2,-2i}\ZZ_2&\ar[l]_{2}\MM^S_{-2,-2i}\ZZ_2&\ar[l]_{\partial}
    \MM^S_{-1,-2i}\FF_2\ar@{=}[d]&\ar[l]_{q}\MM^S_{-1,-2i}\ZZ_2\ar@{=}[d]&\ar[l]_{2}\MM^S_{-1,-2i}\ZZ_2\ar@{=}[d]&\ar[l]_{\partial}\MM^S_{0,-2i}\FF_2\ar@{=}[d]\\
    \FF_2\{\rho^2\tau^{2i-2}\}&?&?&\FF_2\{\rho\tau^{2i-1},\epsilon\tau^{2i-1}\}&\ZZ/w_{2i}\{\epsilon_{2i}\}&\ZZ/w_{2i}\{\epsilon_{2i}\}&\FF_2\{\tau^{2i}\}.
    }
    $$ 
    We know that $\beta(\rho\tau^{2i-1})=\rho^2\tau^{2i-2}$, so $q\colon\MM^S_{-2,-2i}\ZZ_2\to\MM^S_{-2,-2i}\FF_2=\FF_2\{\rho^2\tau^{2i-2}\}$ must be a bijection and (we are free to choose the generator $\epsilon_{2i}\in \MM^S_{-1,-2i}\ZZ_2$ so that) $q(\epsilon_{2i})=\epsilon\tau^{2i-1}$.  Since $\epsilon^2=0$ this also gives the relation $\epsilon_{2i}\epsilon_{2j}=0$.  In odd weight $k=-2i-1<0$ we likewise consider
    $$\xymatrixcolsep{.5pc}\xymatrix{
    \MM^S_{-2,k}\FF_2\ar@{=}[d]&\ar[l]_{q}\MM^S_{-2,k}\ZZ_2&\ar[l]_{2}\MM^S_{-2,k}\ZZ_2&\ar[l]_{\partial}
    \MM^S_{-1,k}\FF_2\ar@{=}[d]&\ar[l]_{q}\MM^S_{-1,k}\ZZ_2\ar@{=}[d]&\ar[l]_{2}\MM^S_{-1,k}\ZZ_2\ar@{=}[d]&\ar[l]_{\partial}\MM^S_{0,k}\FF_2\ar@{=}[d]\\
    0&?&?&\FF_2\{\rho\tau^{2i},\epsilon\tau^{2i}\}&\ZZ_2\{\rho_k,\epsilon_{k}\}/2\rho_k&\ZZ_2\{\rho_k,\epsilon_{k}\}/2\rho_k
    &\FF_2\{\tau^{2i+1}\}.
    }
    $$
    We immediately get that $\MM^S_{-2,-2i-1}\ZZ_2=0$ and all products of classes in degree $1$ with different parity must vanish: $\epsilon_{2j}\epsilon_{2i+1}=\epsilon_{2j}\rho_{2i+1}=0$.  Now,  by \cite[Lemma~3.3 and 3.4]{MR1761621} (there is a natural choice of generators $\rho_{2i+1}$ and $\epsilon_{2i+1}$ so that) $q\colon\MM^S_{-1,-2i-1}\ZZ_2\to\MM^S_{-1,-2i-1}\FF_2$ sends $\rho_{2i+1}$ to $\rho\tau^{2i}$ and $\epsilon_{2i+1}$ to $\epsilon\tau^{2i}$ and so we get the remaining products of classes in degree $-1$: 
$\rho_{2i+1}\rho_{2j+1}=\rho_1\rho_{2(i+j)+1}$ is the unique generator in bidegree $(-2,-2i-2j-2)$ and $\epsilon_{2i+1}\epsilon_{2j+1}=0$.

    To finish off the result, we use that the ring map $q\colon\MM^S_{*}\ZZ_2\to\MM^S_{*}\FF_2$ is an isomorphism in all even bidegrees $(d,w)$ with $d<-2$.
   This follows via base-change to the rational and real numbers,
    see \cite[Lemma 7.8 and Theorem~7.9]{MR4192569}, 
    and uses that $\MM^\RR_*\ZZ_2\to\MM^\RR_*\FF_2$ is the map given by $\ZZ_2[\rho,\tau^2]/2\rho\to\FF_2[\rho,\tau^2]\subseteq\FF_2[\rho,\tau]$.
\end{proof}

Since $\tau$ is not available as a generator in $\mssz_2$, but its multiples with $\tau_i$'s and $\xi_i$'s appear in the kernel  of $\ker\beta_{\MM_*^S\FF_2(\MM^S\ZZ)}\to\MM^S_*\FF_2$, we introduce the notation $\tau_{j,i}$ and $\xi_{j,i}$ for the elements in $\MM^S_*\ZZ_2\times_{\MM^S_*\FF_2}(\MM^S_*\FF_2)\MM^S\ZZ$ mapping to $\tau^i\tau_j$ and $\tau^i\xi_j$ in 
$\MM^S_*\FF_2(\MM^S\FF_2)$, respectively.  
We retain the notation $\eta_{a,U}$ and $y_{a,U}$ with the same meaning as before (with $\tau_j$ and $\xi_i$ replaced by $\tau_{j,0}$ and $\xi_{j,0}$ respectively) and obtain

\begin{lemma}
  \label{lemma:Z12}
  Over $S=\Spec(\ZZ[1/2])$, the $2$-adic dual Steenrod algebra $\MM_*^S(\ZZ)(\MM^S\ZZ)_2$ is generated by
  $\{y_{a,U}, \rho\eta_{a,U}+\tau y_{a,U}\}$ as a
  $\MM^S_*\ZZ_2$-subalgebra of
  $$\MM^S_*\ZZ_2\times_{\MM^S_*\FF_2}(\MM^S_*\FF_2)\MM^S\ZZ\cong
  \ZZ_2[\rho_{2i+1},\epsilon_{i+1},\tau_{j+1,i},\xi_{j+1,i}]_{i,j\geq0}/\left(
  \begin{smallmatrix}
    2\rho_{2i+1},\quad w_{2i}\epsilon_{2i},\\
    \rho_{2i+1}\epsilon_j,\quad\epsilon_i\epsilon_j,\\
    \rho_{2i+1}\rho_{2j+1}+\rho_1\rho_{2(i+j)+1}\\
    2\tau_{i,j},\quad 2\xi_{i,j}\\
    \rho_{2i+1}\tau_{j,k}+\rho_1\tau_{j,2i+k},\qquad\epsilon_i\tau_{j,k}+\epsilon_1\tau_{j,j-1+k},\\
    \rho_{2i+1}\xi_{j,k}+\rho_1\xi_{j,2i+k},\qquad\epsilon_i\xi_{j,k}+\epsilon_1\xi_{j,j-1+k}
  \end{smallmatrix}
\right).
  $$
\end{lemma}

\begin{footnotesize}
\bibliographystyle{plain}
\bibliography{mhh}
\end{footnotesize}
\vspace{0.1in}

\begin{center}
Department of Mathematics, University of Bergen, Norway\\
email: dundas@math.uib.no
\end{center}

\begin{center}
Department of Mathematics, University of Milan, Italy \& \\
Department of Mathematics, University of Oslo, Norway\\
e-mail: paul.oestvaer@unimi.it, paularne@math.uio.no
\end{center}
\end{document}